\documentclass[11pt, a4paper,reqno]{amsart}
\usepackage[a4paper, margin=4cm]{}
\usepackage{amssymb,amsfonts, amsthm, xfrac, amscd}
\usepackage{setspace}
\usepackage{breqn}
\usepackage{times}
\usepackage{changes}
\usepackage{adjustbox}
\usepackage{color}
\usepackage{array}
\usepackage{times}

\usepackage[utf8]{inputenc} 
\usepackage{changes}
\usepackage{color}
\usepackage[hidelinks]{hyperref}

\allowdisplaybreaks[4]

\newtheoremstyle{myremark}     {10pt}{10pt}{}{}{\bfseries}{.}{.5em}{}

\newtheorem{thm}{Theorem}[section]
\newtheorem{cor}[thm]{Corollary}
\newtheorem{lem}[thm]{Lemma}
\newtheorem{pro}[thm]{Proposition}
\theoremstyle{definition}
\newtheorem{defn}[thm]{Definition}

\theoremstyle{myremark}
\newtheorem{rem}[thm]{Remark}
\usepackage{amsthm}

\numberwithin{equation}{section}
\begin{document}

	\title[Variance gamma approximation]{Variance gamma approximation to sums of triplewise independent random variables}

	 \author[Panda and Barman]{Aditi Panda and Kalyan Barman}
     \address{\hskip-\parindent
			Aditi Panda, Department of Mathematics, NIT Warangal,
		Warangal - 506004, India.}
		\email{ap25mar2r05@students.nitw.ac.in}
        
	 \address{\hskip-\parindent
		Kalyan Barman, Department of Mathematics, NIT Warangal,
		Warangal - 506004, India.}
	
	 \email{barmankalyan@nitw.ac.in}



	
	 \begin{abstract}
In this article, we first discuss how triplewise independent random variables (rvs) are connected to a complete bipartite graph. Using the connection, we construct a sequence of triplewise independent rvs. We next consider a variance gamma (VG) approximation of sums of such triplewise independent rvs.  Using Stein's method and the generalized zero-bias transformation, we obtain our bounds. Related limit theorems are also discussed.
	 \end{abstract}

	\maketitle
	
\section{ Introduction} \label{Intro}
\noindent For a sequence of mutually independent and identically distributed (i.i.d.) random variables (rvs) $X_1, X_2,\ldots, X_n$ with $\mathbb{E}[X_i] = \mu$ and $\mathrm{Var}(X_i)=\sigma^2$, for $1\leq i \leq n$, where $0<\sigma<\infty$, it is known that the standardized partial sums
\begin{equation}
S_n := \frac{1}{\sigma\sqrt{n}}
\left(\sum_{k=1}^{n} X_k - n\mu \right) \overset{d}{\to} \mathcal{N}(0,1), \text{ as } n \to \infty,
\end{equation}
where $\overset{d}{\to}$ denotes the convergence in law. This result is known as the
Lindeberg-L\'evy central limit theorem (CLT); see \cite{Lindeberg1992} and \cite{Levy1925}. It is also known that mutual independence in general cannot be relaxed to the weaker notion of pairwise independence; see \cite{Avanzi2021}. Moreover, it cannot even be relaxed to triplewise independence; see \cite{BMO2021}. This article mainly focuses on approximating the partial sums of triplewise independent rvs.

\vskip 1ex
\noindent In general, the $K$-tuplewise independence is defined as follows (see, Definition 1 of \cite{Raik2025}).
\begin{defn}\label{Def:triplewise}
   Let $K\in\{2, 3, 4,\ldots\}$. An indexed family of rvs
$X_i$, $i\in I$, is $K$-tuplewise independent if the rvs
$X_{i_1},X_{i_2},\ldots,X_{i_K}$ are mutually independent for any
$K$-tuple of distinct indices $i_1,i_2,\ldots,i_K$. 
\end{defn}
 \noindent Counterexamples to the CLT can be traced back to \cite{RW1964}, who constructed a sequence of pairwise i.i.d. rvs taking the values $1$ and $-1$ with equal probabilities. By Theorem $1$ of \cite{RW1964}, the absolute values of the partial sums of that sequence can be bounded by a fixed random variable (rv), so that their standardized counterparts $S_n$ converge in distribution to zero. Pruss \cite{Pruss1998} constructs a counterexample to the CLT which is a sequence of $K$-tuplewise i.i.d. rvs, where $K$ can be arbitrary, and the marginal distribution can be any symmetric distribution with finite variance. Bradley and Pruss \cite{BradleyPruss2009} construct a sequence of $K$-tuplewise i.i.d rvs, which is strictly stationary. Recently, Avanzi {\it et al.} \cite{Avanzi2021} provide a survey of further constructions and construct a broad family of counterexamples for pairwise independence. Later, Beaulieu {\it et al.} \cite{BMO2021} modify the construction of Avanzi {\it et al.} \cite{Avanzi2021} to one which is based on a suitable sequence of graphs, each graph giving a family of $K$-tuplewise i.i.d. rvs. The rvs obtained from all graphs can be arranged into an array, each graph giving one row. They provide an increasing sequence of
graphs giving triplewise independent rows and standardized row sums converging in law to a variance-gamma (VG) distribution, which is not normal, see Subsection 4.1 of Beaulieu {\it et al. \cite{BMO2021}}. From that array, a sequence can be extracted such that its
standardized partial sums do not converge to a normal distribution because it has a subsequence that converges to the VG distribution. 

\vskip 1ex
\noindent In this article, we consider the VG approximation to the partial sums of triplewise independent rvs. The present article is, to the best of our knowledge, the first work to derive  bounds with respect to a probability metric for triplewise independent rvs.  Using Stein’s method and the generalized zero-bias transformation, we obtain error bounds for the approximation problem considered.

\vskip 1ex
\noindent The organization of the article is as follows. In Section \ref{Preliminary}, we discuss some preliminary results, which will be useful later. In Section \ref{pre}, we discuss the VG distribution and its related results. In Section \ref{MainResults}, we discuss our approximation results. We derive error bounds for the VG approximation to the partial sums of triplewise independent rvs.

\section{Notations and Preliminary Results}\label{Preliminary}
\noindent In this section, we introduce some notations and discuss important preliminary results. Let us first recall the definition of a complete bipartite graph ( see, for instance, \cite{BroHae2011} for more details).
\begin{defn}
    A complete bipartite graph is a special type of bipartite graph where the set of vertices is partitioned into two disjoint sets and each vertex of the first set gives an edge to each vertex of the second set.
\end{defn}
\noindent Next, we see how the triplewise independent rvs are connected to a complete bipartite graph (see, \cite{BMO2021} for more detail). Consider a complete bipartite graph ${G_m}$ with $m$ vertices in each set and $m^2$ edges. Let $\{v_1, v_2,\dotsc, v_m\}$ be the vertices in first set and $\{w_1, w_2, \dotsc, w_m\}$ be the vertices in second set. Let $n = m^2$. Let $M_1, M_2,..., M_{2m}$ be the sequence of i.i.d discrete uniform rvs on the set $\{1,2,\dotsc, l\}$. More precisely, $
 \mathbb{P}(M_j = i) = \frac{1}{l}, \text{ for all } i \in \{1,2,\dotsc, l\}.
 $ We now assign the rvs $M_1, M_2,..., M_m$ to each vertex of first set of vertices of $G_m$. We also assign the rvs $M_{m+1},\dotsc, M_{2m}$ to the second set of vertices of $G_m$. For every pair $1 \leq i \leq m$ and $m+1 \leq j \le 2m$, let us define
$$
D_{i,j} :=
\begin{cases}
1, & \text{if } M_i = M_j, \\
0, & \text{otherwise}.
\end{cases}
$$
Note that $D_{i,j}$ gives the edges that connect $M_i$ and $M_j$. Let $n$ be the total number of edges. We relabel the rvs $\{D_{i,j}\}$, ${1\leq i \leq j \leq n}$, as $D_1, D_2, \dotsc, D_n$. By Remark 2.1 of \cite{BMO2021} and Definition \ref{Def:triplewise}, we observe that the sequence $\{D_1, D_2, \dotsc, D_n\}$ is triplewise independent. Under the above setup, let us define
\begin{equation}\label{def:xin}
 W := \frac{\sum_{k=1}^n D_k - nl^{-1}}{\sqrt{nl^{-1}(1 - l^{-1})}}.
\end{equation}
\noindent The following theorem provides the asymptotic distribution of $W$ (see, Section 4.1 of \cite{BMO2021} for more detail).

\begin{thm} \label{Thm:BPG}
    Let $\{G_m\}_{m \geq 1}$ be the sequence of bipartite graphs, where i.i.d. discrete uniform rvs $M_1, \dotsc, M_{2m}$ are assigned to the vertices of $\{G_m\}$. That is, $ M_1, \dotsc, M_m$ are assigned to the $ m$ vertices of first set, and $M_{m+1}, \ldots, M_{2m}$ to the $m$ vertices of second set. Then, $W \overset{d}{\to} \frac{Z}{\sqrt{l-1}}$, as $m\to \infty$ (or, equivalently, as $n \to \infty$), where $Z \sim$ VG$(l-1,0,1,0)$ and VG denotes the variance-gamma distribution (see Definition \ref{Def:vg}).
\end{thm}
\noindent Next, we define Wasserstein and bounded Wasserstein distances (see \cite{GauntLi2023}). Consider the function space
\begin{align}\label{fs1}
\mathcal{H}_W &= \left\{h:\mathbb{R}\to \mathbb{R}~\big|~h \text{ is 1-Lipschitz and }|h(x)-h(y)|\leq |x-y|  \right\},
\end{align} 
\noindent
where $\|h\|=\sup_{x\in\mathbb{R}}|h(x)|$. Then, for any two random variables $Y$ and $Z$,
\begin{equation}\label{smwdis}
\sup_{h \in \mathcal{H}_W}\left|\mathbb{E}[h(Y)]-\mathbb{E}[h(Z)]\right|:=d_{W}(Y,Z)~(\text{say})
\end{equation}
\noindent
is called the Wasserstein distance. Also, when we choose the function space
\begin{align*}
	\mathcal{H}_{bW} &= \left\{h:\mathbb{R}\to \mathbb{R}~\big|~h \text{ is 1-Lipschitz, }\|h\| \leq 1 \text{ and }\|h^{\prime}\| \leq 1 \right\},
\end{align*}
\noindent so that $d_{\mathcal{H}_{bW}} :=d_{bW}$ (say) is called the bounded Wasserstein distance. 

\vskip 1ex
\noindent Next, we discuss components of Stein’s method. In general, the method is based on the fact that, any real-valued rv \(Z\) has a distribution \(F_Z\) if and only if there exists an operator \(A\) (also called the Stein operator) such that $\mathbb{E}(Af(Z))=0,$
where \(f\in\mathcal{F}\) (a suitable function space). This characterization leads us to the Stein equation
\begin{align}\label{STEQ:gen}
    Af(x)=h(x)-\mathbb{E}h(Z),
\end{align}
\noindent where \(h\) is a real-valued test function. Replacing \(x\) with a rv \(Y\) and taking expectations on both sides of \eqref{STEQ:gen} gives
\begin{align}\label{STEQ:gen1}
    \mathbb{E}h(Y)-\mathbb{E}h(Z)=\mathbb{E}(Af(Y)).
\end{align}
\noindent The equality \eqref{STEQ:gen1} plays a crucial role in Stein’s method. For a real valued test function \(h\), the problem of bounding the quantity $|\mathbb{E}h(Y)-\mathbb{E}h(Z)|$
relies on the bounds for the solution of \eqref{STEQ:gen} and behavior of \(Y\). For more details on Stein’s method, we refer to the reader \cite{ArrasHoudre2019,BarUp2024,ANV2022,UpBar2022} and the references therein.

\section{Variance gamma distribution and related results}\label{pre}
\noindent In this section, we discuss some important results related to the VG distribution. Let us first recall the definition of a VG distribution (see \cite{Gaunt2014,Gaunt2025}).
\begin{defn}\label{Def:vg}
    The VG distribution with parameters $r>0$, $\theta\in\mathbb{R}$, $\sigma>0$, $\mu\in\mathbb{R}$ has probability density function
\begin{equation}\label{den:VG}
p(x)=\frac{1}{\sigma\sqrt{\pi}\,\Gamma\!\left(\frac{r}{2}\right)}
e^{\frac{\theta}{\sigma^{2}}(x-\mu)}
\left(\frac{|x-\mu|}{2\sqrt{\theta^{2}+\sigma^{2}}}\right)^{\frac{r-1}{2}}
K_{\frac{r-1}{2}}
\left(
\frac{\sqrt{\theta^{2}+\sigma^{2}}}{\sigma^{2}}|x-\mu|
\right),
\end{equation}
with support $\mathbb{R}$. Here $K_{\nu}(x)$ is a modified Bessel function of the second kind, defined by $K_{\nu}(x)=\int_{0}^{\infty} e^{-x\cosh(t)}\cosh(\nu t)\,dt.$ For a rv $Z$ with density \eqref{den:VG}, we write $Z\sim \mathrm{VG}(r,\theta,\sigma,\mu)$.
\end{defn}  
\noindent Many probability distributions such as Laplace, product-normal, gamma, normal, among many others, belong to the VG family. For more details, we refer the reader to \cite{Gaunt2014,Gaunt2022,Gaunt2025}, and the references therein.

\vskip 1ex
\noindent Next, we discuss Stein's method for the VG distribution. The following proposition gives a Stein equation for the VG distribution (see \cite{Gaunt2022}).

\begin{pro}
    Let $Z\sim \mathrm{VG}(r,\theta,\sigma,\mu)$ with density given in \eqref{den:VG}. Then a Stein equation for $Z$ is given by
    \begin{align}\label{gen stein equn}
        \sigma^{2}(x-\mu)f^{\prime\prime}(x)+\left(\sigma^{2}r+2\theta(x-\mu)\right)f^\prime(x)+\left(r\theta-(x-\mu)\right)f(x)=h(x)-\mathbb{E}(h(Z)),
    \end{align}
    where $h\in \mathcal{H}$ is a real-valued test function.
\end{pro}
\noindent The following corollary gives the Stein equation for $\mathrm{VG}(1,0,1,0)$ distribution. Note that the $\mathrm{VG}(1,0,1,0)$ distribution corresponds to the distribution of the product of two independent standard normal $\mathcal{N}(0,1)$ rvs, whereas the $\mathrm{VG}(n,0,1,0)$ distribution arises as a sum of $n$ independent copies of such rvs (see Theorem 1 of \cite{Gaunt2019}). 
\begin{cor}
        Let $Z\sim \mathrm{VG}(1,0,1,0)$ with density given in \eqref{den:VG}. Then a Stein equation for $Z$ is given by
    \begin{align}\label{SteinforZ}
         xf^{\prime\prime}(x)+f^\prime(x)-xf(x)=h(x)-\mathbb{E}(h(Z)),
    \end{align}
    where $h\in \mathcal{H}$ is a real-valued test function.
\end{cor}
\noindent The next result gives the solution to the Stein equation \eqref{SteinforZ}, which essentially follows from Lemma 2.3 of \cite{Gaunt2017}. Let \(C^k_b(\mathbb{R})\) be the space of $k$-times continuously differentiable functions on \(\mathbb{R}\) that, along with all their derivatives up to order $k$, are bounded.
\begin{pro}
    Let $h \in C^1_b(\mathbb{R}) $. Then the unique bounded solution 
$f_h : \mathbb{R} \to \mathbb{R}$ to the Stein equation \eqref{SteinforZ} is given by
\begin{equation}\label{Sol:SE}
f_h(x)=- K_0(|x|)\int_{0}^{x}I_0(y)
 \tilde{h}(y) dy 
- I_0(x)
\int_{x}^{\infty}
K_0(|y|)\tilde{h}(y)
dy,
\end{equation}
where $\tilde{h}(y)=h(y) - \mathbb{E}h(Z)$, $I_0(x)$ and $K_0(x)$ are modified Bessel functions, defined, for all $x \in \mathbb{R}, \nu \in \mathbb{R}$, by 
$$
I_{\nu}(x) = \sum_{k=0}^{\infty}
\frac{1}{\Gamma(\nu +k +1)k!}
\left(\frac{x}{2}\right)^{\nu + 2k} \text{ and }K_{\nu}(x)= \frac{\pi}{2 sin(\nu \pi)}(I_{-\nu}(x) - I_{\nu}(x)).
$$
\end{pro}
\noindent Next, we discuss the properties of the solution to the Stein equation \eqref{Sol:SE}. The derivation of the following properties follows from Theorem 2.1 and Lemma 2.4 of \cite{Gaunt2017}.
\begin{lem}
    Suppose that $h \in C_b^3(\mathbb{R})$. Let $f_h$ be defined in \eqref{Sol:SE}. Then,
\[
\|f\| = 3
\|\tilde{h}
\|, ~\|f^{\prime}\|=\frac{3}{2}
\|\tilde{h}\|,~\|f^{\prime \prime}\|=2\|h^{\prime}\|+
5
\|\tilde{h}\|,
\]
\[
\|f^{(3)}\|=4\|h^{\prime \prime}\|
+5\|h^{\prime}\|+
4.89\|\tilde{h}\|,
\|f^{(4)}\|=8\|h^{(3)}\|+
9\|h^{\prime \prime}\|+
6.81\|h^{\prime}\|+
15.75
\|\tilde{h}
\|,
\]
and we also have
\[
\|x f(x)\|\le2\|\tilde{h}\|,~\|x f^{\prime}(x)\|\le\frac{3}{2}
\|\tilde{h}\|,~\|x f^{\prime \prime}(x)\| \le \frac{9}{2}
\|\tilde{h}\|,
\]
where $\|f\|=\sup_{x\in \mathbb{R}}|f(x)|$ and $\tilde{h}(x)=h(x)-\mathbb{E}(h(Z)),$ and $ Z\sim VG(1,0,1,0)$.
\end{lem}

\begin{lem}
    Let $f$ be the solution of $Z$ of the Stein equation \eqref{SteinforZ}. Also, let $A_2f(x):=xf^{\prime\prime}(x)+f^\prime(x)$, $x\in \mathbb{R}$. Then 
    \begin{equation} \label{def:A_2f prime}
 \left\| (A_2f)^{\prime} \right\|
\leq
\left\| h^{\prime} \right\|
+  \frac{9}{2}\left\| \tilde{h}(x) \right\| \text{ and } \left\| (A_2 f)^{\prime \prime} \right\|
\leq
\left\| h^{\prime \prime} \right\|
+  \frac{15}{2}\left\|\tilde{h}(x) \right\|.
\end{equation}
\end{lem}
\begin{rem}
Note that the quantity $\|(A_2f)^{\prime}\|$ in \eqref{def:A_2f prime} can alternatively be bounded using  the inequality (3.16) of \cite{Gaunt2020} as 
\begin{equation}\label{improved bound}
\|(A_2f)^{\prime}\| \leq \frac{99}{8}\|h^{\prime}\|.
\end{equation}
The advantages of this bound are discussed in Theorems \ref{thm:error1} and \ref{thm:err2}. Indeed, \eqref{improved bound} leads us to obtain the error bounds for the Wasserstein distance for the VG approximation to sums of triplewise independent rvs.
\end{rem}
\section{Main results}\label{MainResults}
\noindent In this section, we discuss our main results and their relevance to the literature. Before stating our results, we need the following setup.

\subsection{The setup} Let $\{G_m\}_{m \geq 1}$ be the sequence of bipartite graphs. Let $M_1, \dotsc, M_{2m}$ be a sequence of i.i.d. discrete uniform rvs that are assigned to the vertices of $\{G_m\}$. That is, $ M_1, \dotsc, M_m$ are assigned to the $ m$ vertices of first set, and $M_{m+1}, \dotsc, M_{2m}$ to the $m$ vertices of second set. Observe next that the number of vertices = $2m$ and the number of edges $n = m^2$. Define 
\[ 
N_i^{(1)}:={N_i}^{(1)}(m) = \text {the number of $M_j$'s equal to $i$ within the sample $\{M_j\}_{j=1}^{m}$}, \text{and}
\]
\[
N_i^{(2)}:={N_i}^{(2)}(m) = \text {the number of $M_j$'s equal to $i$ within the sample $\{M_j\}_{j=m+1}^{2m}$}.
\]
\noindent Then \textbf{$N^{(p)}$}$=({N_1}^{(p)}, \dotsc, {N_l}^{(p)}) \sim \text{Multinomial}(m,(\frac{1}{l}, \dotsc, \frac{1}{l})$ for $p=1,2$. Moreover \textbf{$N^{(1)}$} and \textbf{$N^{(2)}$} are independent. Let
\begin{align*}
I_{j} &:= 1_{\{{M_j = i\}}}
 =\begin{cases}
1, & \text{ if }  M_j = i, \\
0, & \text{otherwise.}    
\end{cases}
\end{align*}
\noindent By definition, we have ${N_i}^{(1)}= \sum_{j=1}^m I_j \text{  and  } {N_i}^{(2)}= \sum_{j=m+1}^{2m} I_j$, since $I_j \sim \text{Bernoulli} (\frac{1}{l}).$

\noindent Hence, we can write \eqref{def:xin} as
\begin{equation} \label{def:Xi_n}
 W = \frac{\sum_{i=1}^l {N_i}^{(1)}{N_i}^{(2)} - nl^{-1}}{\sqrt{nl^{-1}(1 - l^{-1})}}.
\end{equation}
 Note that $\mathbb{E}[{N_i}^{(1)}] = \mathbb{E}[{N_i}^{(2)}]= \frac{m}{l} $, $\mathrm{Var}({N_i}^{(1)})= \mathrm{Var}({N_i}^{(2)}) = \frac{m}{l}(1-\frac{1}{l})$ and $\mathbb{E}[{N_i}^{(1)}{N_i}^{(2)}] = \frac{m^2}{l^2}$.
\vskip 1ex
\noindent Define, 
\begin{equation}\label{def:X_i,Y_i}
X_i := \frac{{N_i}^{(1)} -  \frac{m}{l}}{\sqrt {\frac{m}{l}(1-\frac{1}{l})}} \qquad \text{ and } \qquad Y_i := \frac{{N_i}^{(2)} -  \frac{m}{l}}{\sqrt {\frac{m}{l}(1-\frac{1}{l})}},
\end{equation}
such that $\mathbb{E}[X_i] =  \mathbb{E}[Y_i] =0\text{ and } \mathrm{Var}(X_i) =\mathrm{Var}(Y_i) = 1.$

\vskip 1ex
\noindent Let 
\begin{equation}\label{def:X'_i,Y'_i}
   X^{\prime}_i := {N_i}^{(1)} -  \frac{m}{l} \qquad
   \text{ and } \qquad Y^{\prime}_i := {N_i}^{(2)} -  \frac{m}{l}.
\end{equation}
Hence from  \eqref{def:X_i,Y_i} and \eqref{def:X'_i,Y'_i} we get 
\begin{equation} \label{def:{N_i}^{(1)}{N_i}^{(2)}}
{N_i}^{(1)}{N_i}^{(2)} = \left(X^{\prime}_i + \frac{m}{l}\right) \left(Y^{\prime}_i + \frac{m}{l}\right) = X^{\prime}_iY^{\prime}_i + \frac{m}{l}( X^{\prime}_i + Y^{\prime}_i) +  \frac{m^2}{l^2}.
\end{equation}
So, from  \eqref{def:Xi_n} and \eqref{def:{N_i}^{(1)}{N_i}^{(2)}} we can write
\begin{align}\label{eqnnum1}
 \sum_{i=1}^l {N_i}^{(1)}{N_i}^{(2)}
    &= \sum_{i=1}^l X^{\prime}_iY^{\prime}_i + \sum_{i=1}^l \frac{m}{l}( X^{\prime}_i + Y^{\prime}_i) +  \frac{m^2}{l}.
\end{align}
\noindent Note that
\begin{equation}\label{eqnnum2}
    \sum_{i=1}^l X^{\prime}_i =  \sum_{i=1}^l ({N_i}^{(1)} -  \frac{m}{l}) = \sum_{i=1}^l {N_i}^{(1)} -  \sum_{i=1}^l  \frac{m}{l}
    = m - m
   = 0 \text{ and }\sum_{i=1}^l Y^{\prime}_i = 0.
\end{equation}

\vskip 1ex
\noindent From \eqref{def:X_i,Y_i} and \eqref{def:X'_i,Y'_i}, we have
\begin{equation}\label{eqnnum3}
     \sum_{i=1}^l X^{\prime}_iY^{\prime}_i =  \sum_{i=1}^l   \frac{m}{l}(1- \frac{1}{l}) X_iY_i
     =  \frac{m}{l}(1- \frac{1}{l})  \sum_{i=1}^l  X_iY_i.
\end{equation}
\noindent Using \eqref{eqnnum3} and \eqref{eqnnum2} in \eqref{eqnnum1}, we have

\begin{align}\label{eqnnum4}
    \sum_{i=1}^l {N_i}^{(1)}{N_i}^{(2)}
    &= \frac{m}{l}(1- \frac{1}{l})  \sum_{i=1}^l  X_iY_i+\frac{m^2}{l}.
\end{align}
\noindent Using \eqref{eqnnum4} in \eqref{def:Xi_n}, noting that $n=m^2$, we get
\begin{align}
    W = \sqrt{\frac{1}{l}\left(1-\frac{1}{l}\right)}\sum_{i=1}^lX_iY_i.
\end{align}
\noindent We now have the following lemmas, which will be used later.
\vspace{-0.5cm}
\begin{lem}
    Let $X_i$ and $Y_i$ be two rvs as defined in \eqref{def:X_i,Y_i}. Then
    \begin{equation}
 \mathbb{E} [X_i^{3}]=\mathbb{E} [Y_i^{3}] = \frac{1 - \frac{2}{l}}{\sqrt{\frac{m}{l}(1-\frac{1}{l})}}, \qquad l\ge 2 .       
    \end{equation}
    \noindent Moreover, 
    
    \begin{align}\label{bdtm1}
\mathbb{E} [|X_i^{3}|]= \mathbb{E} [|Y_i^{3}|] \leq \sqrt{\frac{l}{m}}, \qquad l\ge 2.
    \end{align}
\end{lem}
\vspace{ -.7cm}
\begin{proof}
Note that,
\begin{equation}
 \mathbb{E} [|X_i|^{3}] =  \mathbb{E} \Big[\Big(\frac{{N_i}^{(1)} -  \frac{m}{l}}{\sqrt {\frac{m}{l}(1-\frac{1}{l})}}\Big)^3\Big]
 = \frac{1}{\big(\frac{m}{l}(1-\frac{1}{l})\big)^{3/2}} \mathbb{E}\left[\left({N_i}^{(1)} -  \frac{m}{l}\right)^3\right],
\end{equation}
\noindent and
\begin{equation}\label{def:N_i}
{N_i}^{(1)} = I_1 + I_2 + \dotsc + I_m.
\end{equation}
\noindent Therefore,
\[
{N_i}^{(1)} - \frac{m}{l} = \sum_{j=1}^m (I_j - \frac{1}{l}).
\]
\noindent Hence,
\begin{align*}
   \mathbb{E} \Big[\Big({N_i}^{(1)} -  \frac{m}{l}\Big)^3\Big] 
&= \mathbb{E}[({N_i}^{(1)})^3]
-3\frac{m}{l}\,\mathbb{E}[({N_i}^{(1)})^2]
+3\frac{m^2}{l^2}\,\mathbb{E}[{N_i}^{(1)}]
- \frac{m^3}{l^3}.
\end{align*}
Using \eqref{def:N_i}, we have
$$
({N_i}^{(1)})^2=\bigg(\sum_{j=1}^m I_j\bigg)^2.
$$
Hence,
\[
({N_i}^{(1)})^2
=
\sum_{j=1}^m I_j^2
+
\sum_{j\neq k} I_j I_k.
\]
Since \(I_j^2=I_j\),
\begin{align}\label{Def:NI}
   ({N_i}^{(1)})^2
=
\sum_{j=1}^m I_j
+
\sum_{j\neq k} I_j I_k. 
\end{align}
\noindent Taking the expectation on both sides of \eqref{Def:NI}, we get
\[
\mathbb{E}[({N_i}^{(1)})^2]
=
\sum_{j=1}^m \mathbb{E}[I_j]
+
\sum_{j\neq k} \mathbb{E}[I_j I_k].
\]
So,
\[
\mathbb{E}[({N_i}^{(1)})^2] = \frac{m}{l} + m(m-1)\frac{1}{l^2}.
\]
\noindent Using \eqref{def:N_i}, we write
\[
({N_i}^{(1)})^3 =\sum_{j=1}^m I_j^3 
=
\sum_{j=1}^m I_j^3
+
3\sum_{j\neq k} I_j^2 I_k
+
\sum_{\substack{j,k,l\\ \text{distinct}}} I_j I_k I_l.
\]
Since \(I_j^2=I_j\) and \(I_j^3=I_j\),
\begin{align}\label{Def:NI1}
   ({N_i}^{(1)})^3
= \sum_{j=1}^m I_j
+ 3\sum_{j\neq k} I_j I_k + \sum_{\substack{j,k,l\\ \text{distinct}}} I_j I_k I_l. 
\end{align}

\noindent Taking the expectation on both sides of \eqref{Def:NI1}, we get
\[
\mathbb{E}[({N_i}^{(1)})^3]
=\sum_{j=1}^m \mathbb{E}[I_j] + 3\sum_{j\neq k} \mathbb{E}[I_j I_k]
+ \sum_{\substack{j,k,l\\ \text{distinct}}}
\mathbb{E}[I_j I_k I_l].
\]
Hence, we get
\[
\mathbb{E}[({N_i}^{(1)})^3]
= \frac{m}{l} + 3m(m-1)\frac{1}{l^2}
+ m(m-1)(m-2)\frac{1}{l^3}.
\]
Therefore,
\[
 \mathbb{E} \Big[\Big({N_i}^{(1)} -  \frac{m}{l}\Big)^3\Big] =  \frac{m}{l} (1 -  \frac{1}{l}) (1 - \frac{2}{l}).
\]
Moreover,
\begin{equation}
  \mathbb{E} [|X_i|^{3}] = \frac{1}{\Big(\frac{m}{l}(1-\frac{1}{l})\Big)^{3/2}} \frac{m}{l} (1 -  \frac{1}{l}) (1 - \frac{2}{l}) 
  = \frac{1 - \frac{2}{l}}{\sqrt{\frac{m}{l}(1-\frac{1}{l})}}, \qquad l\ge 2.
\end{equation}
Similarly, we obtain
\begin{equation}
 \mathbb{E} [|Y_i|^{3}] = \frac{1 - \frac{2}{l}}{\sqrt{\frac{m}{l}(1-\frac{1}{l})}}, \qquad l\ge 2.
\end{equation}
To find the upper bound of the third moment of $X_i$ and $Y_i$, we write
\begin{equation}
     \frac{1 - \frac{2}{l}}{\sqrt{\frac{m}{l}(1-\frac{1}{l})}} = \frac{l-2}{\sqrt{m(l-1)}} .
\end{equation}
\noindent Note that,
\begin{align*}
    (l-2)^2 &= l^2 - 4l + 4 \leq l^2 - l = l(l-1), ~l\geq 2.
\end{align*}
\noindent Hence,
\begin{equation}
     \mathbb{E} [|X_i|^{3}] =  \mathbb{E} [|Y_i|^{3}] \leq \sqrt{\frac{l}{m}}, \qquad l\ge 2.
\end{equation}
\noindent This proves the result.
\end{proof}
\noindent We now define the zero-biased distribution of order $n$ (see, \cite[Definition 1.1]{Gaunt2017} for more details).
\begin{defn}
 Let $W$ be a mean-zero rv with finite, non-zero variance $\sigma^2$. We say that $W^{*(n)}$ has the $W$-zero biased distribution of order $n$ if for all $n$ times
differentiable functions $f$ for which $\mathbb{E}[Wf(W)]$ exists,
\[
\mathbb{E}[Wf(W)]
= \sigma^2 \mathbb{E}\!\left[A_n f\left(W^{*(n)}\right)\right],
\]   
where $A_n f(x)=x^{-1}T^{n} f(x)$
and
$Tf(x)=x f^\prime (x)$. Here $T^{n}$ denotes $n$-times compositions of $T$, that is, $T^n = T \circ T \circ \cdots \circ T $ $(n \text{ times}).$
\end{defn}
\noindent Next, we present an important property of a zero-biased distribution similar to \cite[Lemma 1.1]{Gaunt2017}.
\begin{lem}\label{lem:zerobias}
 Let $Z_1, \ldots, Z_l$ be independent mean zero rvs with
$\mathbb{E}Z_i^2 = 1$. Set
$W :=\sum_{i=1}^l Z_i$
and $\mathbb{E}W^2 = l$ .
Let $I$ be a random index independent of the rv $Z_i$ such that $\mathbb{P}(I=i)=\frac{1}{l}.$ Let $W_i :=W-Z_i=\sum_{j\ne i} Z_j .$ Then $W_I+Z_I^{*}$
has the $W$-zero biased distribution, where $Z_I^{*}=Z_I^{*(1)}$ has the $Z_I$-zero biased distribution of order $1$.
\end{lem} 
\begin{proof}
Since $\mathbb{E}Z_i = 0$ and $\mathbb{E}Z_i^2 = 1$ for all $i=1,\dotsc,l$, we have $
\mathbb{E}W^2=\sum_{i=1}^l \mathbb{E}Z_i^2=l.$ Hence, for all smooth functions $f$, we write $\mathbb{E}[Wf(W)]=l\mathbb{E} [f'(W^{*})].$

\noindent Now,
\[
\mathbb{E}[Wf(W)]
=
\mathbb{E}\left[\left(\sum_{i=1}^l Z_i\right)f(W)\right]
=
\sum_{i=1}^l \mathbb{E}\big[Z_i f(W)\big].
\]
Since $W=W_i+Z_i$
and $W_i$ is independent of $Z_i$, applying the zero-bias identity we get
\[
\mathbb{E}\big[Z_i f(W)\big]
=
\mathbb{E}\big[Z_i f(W_i+Z_i)\big]
=
\mathbb{E}Z_i^2\mathbb{E} f'(W_i+Z_i^{*}).
\]
Note that
\[
\mathbb{E}\big[Z_i f(W)\big] = \mathbb{E} f'(W_i+Z_i^{*}),
\]
since $\mathbb{E}Z_i^2=1$. Therefore,
\[
l\mathbb{E} f'(W^{*})
=
\sum_{i=1}^l \mathbb{E} f'(W_i+Z_i^{*}).
\]
Now using the definition of random index and $P(I=i)=\frac{1}{l}$, we obtain
\[
\sum_{i=1}^l \mathbb{E} f'(W_i+Z_i^{*})
=
l \sum_{i=1}^l \frac{1}{l} \mathbb{E} f'(W_i+Z_i^{*})
=
l\mathbb{E} f'(W_I+Z_I^{*}).
\]
\noindent So,
\[
\mathbb{E} f'(W^{*}) = \mathbb{E} f'(W_I+Z_I^{*}).
\]
\noindent This completes the proof.
\end{proof}

\begin{lem} \label{moment of W-W*(2)}
    Let $X_i$ and $Y_i$ be defined in \eqref{def:X_i,Y_i} where $X_i$ follows the same distribution as $X$ and $Y_i$  follows the same distribution as $Y$ for $i = 1, 2, \dotsc, l$. Let $W_l = \frac{1}{\sqrt{l}} \sum_{i=1}^l Z_i$, where $Z_i=X_iY_i$. Then
\begin{align}
   \mathbb{E}\left|W_l - {W_l}^{*(2)}\right|  \leq  \frac{5}{4}\frac{\sqrt{l}}{m}, \qquad l\ge 2. 
\end{align}
\end{lem}
\begin{proof}
Given $W_l = \frac{1}{\sqrt{l}} \sum_{i=1}^l Z_i$. By Lemma \ref{lem:zerobias}, we have
\[
{W_l}^{*(2)} = W_l - \frac{Z_I}{\sqrt{l}} +  \frac{{Z_I}^{*(2)}}{\sqrt{l}}.
\]
Now, 
\[
{W_l}^{*(2)} = \frac{1}{\sqrt{l}}\sum_{i=1}^l Z_i - \frac{Z_I}{\sqrt{l}} +  \frac{{Z_I}^{*(2)}}{\sqrt{l}}.
\]
Therefore,
\[
W_l - {W_l}^{*(2)} = \frac{Z_I}{\sqrt{l}} -  \frac{{Z_I}^{*(2)}}{\sqrt{l}} = \frac{1}{\sqrt{l}}(Z_I - {Z_I}^{*(2)}).
\]
Now,
\[
\mathbb{E}\left|W_l - {W_l}^{*(2)}\right| = \frac{1}{\sqrt{l}}\mathbb{E}\left|(Z_I - {Z_I}^{*(2)})\right| \leq \frac{1}{\sqrt{l}}\Big[\mathbb{E}|Z_I| + \mathbb{E}|{Z_I}^{*(2)}|\Big].
\]
Again we know,
\[
\mathbb{E}|Z_I| = \mathbb{E}|X_IY_I| \leq \sqrt{\mathbb{E}{{X_I}^2}} \sqrt{\mathbb{E}{{Y_I}^2}} = 1,
\]
and
\[
\mathbb{E}|Z_I^{*(2)}| = \mathbb{E}|X_I^{*}Y_I^{*}| = \mathbb{E}|X_I^{*}|\mathbb{E}|Y_I^{*}| = \frac{1}{4}\mathbb{E}|X|^{3}\mathbb{E}|Y|^{3}.
\]
Therefore, 
\begin{equation} \label{expectation inequality}
\mathbb{E}\left|W_l - {W_l}^{*(2)}\right| \leq \frac{1}{\sqrt{l}}\Big[ 1 + \frac{1}{4}\mathbb{E}|X|^{3}\mathbb{E}|Y|^{3} \Big], \qquad l \geq 2.
\end{equation}
\noindent By Holder's inequality for any rv with finite third moment, we have $\left(\mathbb{E}|X|^3\right)^{1/3}
\ge
\left(\mathbb{E}|X|^2\right)^{1/2}.$ Since $\mathbb{E}X^2=1,$ we have $\left(\mathbb{E}|X|^3\right)^{1/3}\ge 1,$ and therefore,
\begin{equation} \label{moment of X}
\mathbb{E}|X|^3 \ge 1.
\end{equation}
\noindent Similarly,
\begin{equation}\label{moment of Y}
    \mathbb{E}|Y|^3 \ge 1.
\end{equation}
\noindent Using \eqref{moment of X} and  \eqref{moment of Y} in \eqref{expectation inequality}, we have 
\begin{align}\label{bdpdtm}
\mathbb{E}\left|W_l - {W_l}^{*(2)}\right| \leq \frac{1}{\sqrt{l}}\mathbb{E}|X|^{3}\mathbb{E}|Y|^{3}\Big[ 1 + \frac{1}{4} \Big]
&= \frac{5}{4}\frac{1}{\sqrt{l}}\mathbb{E}|X|^{3}\mathbb{E}|Y|^{3}, \qquad l \geq 2.
\end{align}
Using \eqref{bdtm1} in \eqref{bdpdtm}, we get
\begin{equation*}
  \mathbb{E}\left|W_l - {W_l}^{*(2)}\right| \leq \frac{5}{4}\frac{\sqrt{l}}{m}, \qquad l \geq 2.  
\end{equation*}
\noindent This proves the result.
\end{proof}
\noindent The following result is a special case of \cite[Theorem 4.1]{Gaunt2017}, for the case $n=2$.
\begin{thm} \label{moment approximation}
    Let $W_l$ be defined as in Lemma \ref{moment of W-W*(2)} and $Z \sim \mathrm{VG(1,0,1,0)}$. Let $f$ be the solution of the $\mathrm{VG}(1,0,1,0)$ Stein equation \eqref{SteinforZ}. Assume that $(W_l,W_l^{*(2)})$ is given on a joint probability space so that
$W_l^{*(2)}$has the $W_l$-zero biased distribution of order $2$.
Then for $h\in C_b^1(\mathbb{R})$, we have
\begin{align}\label{thm:approxres1}
\Big|\mathbb{E}[h(W_l)] - \mathbb{E}[h(Z)]\Big|
\leq
 \left\| (A_2 f)^{\prime} \right\| \mathbb{E}\left| W_l-W_l^{*(2)} \right|.
\end{align}
\end{thm}
\noindent Next, we obtain an error bound for VG approximation to partial sums of triplewise independent rvs. Recently, Beaulieu {\it et al.} \cite[Subsection 4.1]{BMO2021} established a limiting result, which shows that the partial sums of triplewise independent rvs converge to a VG distribution.
\begin{thm}\label{thm:error1}
    Let $W_l$ be defined as in Lemma \ref{moment of W-W*(2)} and $Z \sim \mathrm{VG(1,0,1,0)}$. Assume that $(W_l,W_l^{*(2)})$ is given on a joint probability space so that
$W_l^{*(2)}$has the $W_l$-zero biased distribution of order $2$.
Then 
\begin{align}\label{vgbdbw1}
   d_{bW}(W_l,Z)
\leq
 \frac{5}{4}\frac{\sqrt{l}}{m}\Big[\left\| h^{\prime} \right\|
 + \frac{9}{2}
\left\| \tilde{h}(x) \right\|\Big], \qquad l \geq 2. 
\end{align}
 Also, we have
 
 \begin{align}\label{vgbdbw2}
    d_{W}(W_l,Z)
\leq
 \frac{495}{32}\frac{\sqrt{l}}{m}, \qquad l\geq 2. 
 \end{align}
\end{thm}
\begin{proof}
(i) Using \eqref{def:A_2f prime} and Lemma \ref{moment of W-W*(2)} in \eqref{thm:approxres1}, we get \eqref{vgbdbw1}.

\vskip 1ex
\noindent (ii) Using \eqref{improved bound} and 
 Lemma \ref{moment of W-W*(2)} in \eqref{thm:approxres1}, we get \eqref{vgbdbw2}.
\end{proof}
 \noindent Next, we obtain an error bound in approximating a statistic that has an asymptotic $\mathrm{VG}(l-1,0,1,0)$ distribution.
\begin{thm}\label{thm:err2}
 Let $W_l$ be defined as in Lemma \ref{moment of W-W*(2)} and $Z \sim \mathrm{VG}(l-1,0,1,0)$. Then
\begin{align}\label{Aprroximationfinal}
    d_{bW}(W_l,Z)
\le
 \frac{5}{4}\frac{\sqrt{l}(l-1)}{m}\Big[\left\| h^{\prime} \right\|
+  \frac{9}{2}
\left\| \tilde{h}(x) \right\|\Big], \qquad l \geq 2.
\end{align}
Also, we have
\begin{align}\label{Aprroximationfinalone}
  d_{W}(W_l,Z)
\leq
 \frac{495}{32}\frac{\sqrt{l}(l-1)}{m}, \qquad l \geq 2.  
\end{align} 
\end{thm}
\begin{proof}
 (i)  Let $W_{(i)} := X_iY_i$, so that $ W_l =\frac{1}{\sqrt{l}} \sum_{i=1}^l W_{(i)}$. Using the $\mathrm{VG}(l-1,0,1,0)$ Stein equation \eqref{gen stein equn}, we have 
\begin{align}\label{First:BD}
  \nonumber \left| \mathbb{E}h(W_l)- \mathbb{E} h(Z)  \right| &= \mathbb E\Big[
W_{l}f^{\prime \prime}(W_l)
+ (l-1)f^{\prime}(W_l)
-W_{l}f(W_l)
\Big]  
\\
\nonumber &=  \sum_{i=1}^{l-1} \mathbb{E}\Big[
W_{(i)}f^{\prime \prime}(W_l)
+ f^{\prime}(W_l)
-W_{(i)}f(W_l)
\Big]
\\
\nonumber &= \sum_{i=1}^{l-1}\mathbb{E}\Big[\mathbb{E}[
W_{(i)}f^{\prime \prime}(W_l)
+ f^{\prime}(W_l)
-W_{(i)}f(W_l)\big| (W_{(1)},\ldots,\\
 & \quad\quad W_{(i-1)}, W_{(i+1)},\dots,W_{(l)}]\Big].
\end{align}
\noindent Fix \(i\), and define $V_i:=\sum_{l \neq i}W_{(l)}.$ Then $W_l=W_{(i)}+V_i.$ Next let
\[
\mathcal F_i:=\sigma\big(W_{(1)},\dots,W_{(i-1)},W_{(i+1)},\dots,W_{(l)}\big) \qquad \text{($\sigma$ - algebra)}.
\]
\noindent Hence,
\begin{align}\label{First:BD1}
\nonumber &\mathbb E\Big[
W_{(i)}f''(W_l)+f'(W_l)-W_{(i)}f(W_l)
\,\Big|\,\mathcal F_k
\Big]
\\
&=
\mathbb E\Big[
W_{(i)}f''(W_{(i)}+V_i)
+f'(W_{(i)}+V_i)
-W_{(i)}f(W_{(i)}+V_i)
\Big].
\end{align}
\noindent
Define a shifted function $g_{V_i}(x):=f(x+V_i).$ Then $g_{V_i}^{(m)}(x)=f^{(m)}(x+V_i),$
and therefore,
\begin{align}
\nonumber 
\|g_{V_i}^{(m)}\| &= \sup_{x\in\mathbb R}|f^{(m)}(x+V_i)|
\nonumber\\
&= \sup_{y\in\mathbb R}|f^{(m)}(y)|
\\
&= \|f^{(m)}\|.
\end{align}
\noindent Hence, the derivative norms are invariant under translations. Using \eqref{First:BD1} with $g_{V_i}(x):=f(x+V_i)$ and then applying Theorem \ref{thm:error1},  we get
\begin{align}\label{Stein eqn bound}
  \big|
&\mathbb E[
W_{(i)}g_{V_i}''(W_{(i)})
+g_{V_i}'(W_{(i)})
-W_{(i)}g_{V_i}(W_{(i)})
]
\big| \leq  \frac{5}{4}\frac{\sqrt{l}}{m}\Big[\left\| h^{\prime} \right\|
+  \frac{9}{2}
\left\| \tilde{h}(x) \right\|\Big], \qquad l \geq 2.
\end{align}
\noindent Using \eqref{Stein eqn bound} in \eqref{First:BD}, we get \eqref{Aprroximationfinal}.

\vskip 1ex
\noindent (ii) Following steps similar to case 1, noting that $\|(A_2f)^{\prime}\|$ is bounded as in \eqref{improved bound}, we get \eqref{Aprroximationfinalone}. 

\vskip 1ex
\noindent This completes the proof.
\end{proof}
\begin{rem}
    Note that if $m \to \infty$, then from \eqref{Aprroximationfinal} and \eqref{Aprroximationfinalone}, we have $W_l \overset{\mathcal{d}}{\to} Z$, where $Z \sim \mathrm{VG}(l-1,0,1,0)$, $l \geq 2$. Note that  Beaulieu {\it et al.} \cite{BMO2021} prove that the partial sums of triplewise independent rvs converge to a VG distribution. We obtain the order of convergence $O(m^{-1})$ (or, equivalently, $O(n^{-\frac{1}{2}})$) for this limiting result, which is novel in our opinion. 
\end{rem}


\end{document}